\documentclass{amsart}
\usepackage{amssymb}
\usepackage{amsfonts}
\usepackage{hyperref}

\newtheorem{theorem}{Theorem}
\theoremstyle{plain}

\newtheorem{lemma}{Lemma}

\newtheorem{remark}{Remark}

\numberwithin{equation}{section}

\input{tcilatex}

\begin{document}
\title[Two asymptotic expansions for gamma function]{Two asymptotic
expansions for gamma function developed by Windschitl's formula}
\author{Zhen-Hang Yang and Jing-Feng Tian*}
\address{Zhen-Hang Yang, College of Science and Technology\\
North China Electric Power University, Baoding, Hebei Province, 071051, P.
R. China and Department of Science and Technology, State Grid Zhejiang
Electric Power Company Research Institute, Hangzhou, Zhejiang, 310014, China}
\email{yzhkm\symbol{64}163.com}
\address{Jing-Feng Tian\\
College of Science and Technology\\
North China Electric Power University\\
Baoding, Hebei Province, 071051, P. R. China}
\email{tianjf\symbol{64}ncepu.edu.cn}
\thanks{*Corresponding author: Jing-Feng Tian, e-mail: tianjf\symbol{64}%
ncepu.edu.cn}

\begin{abstract}
In this paper, we develop Windschitl's approximation formula for the gamma
function to two asymptotic expansions by using a little known power series.
In particular, for $n\in \mathbb{N}$ with $n\geq 4$, we have%
\begin{equation*}
\Gamma \left( x+1\right) =\sqrt{2\pi x}\left( \tfrac{x}{e}\right) ^{x}\left(
x\sinh \tfrac{1}{x}\right) ^{x/2}\exp \left( \sum_{k=3}^{n-1}\tfrac{\left(
2k\left( 2k-2\right) !-2^{2k-1}\right) B_{2k}}{2k\left( 2k\right) !x^{2k-1}}%
+R_{n}\left( x\right) \right)
\end{equation*}%
with%
\begin{equation*}
\left\vert R_{n}\left( x\right) \right\vert \leq \frac{\left\vert
B_{2n}\right\vert }{2n\left( 2n-1\right) }\frac{1}{x^{2n-1}}
\end{equation*}%
for all $x>0$, where $B_{2n}$ is the Bernoulli number. Moreover, we present
some approximation formulas for gamma function related to Windschitl's
approximation one, which have higher accuracy.
\end{abstract}

\maketitle

\section{Introduction}

It is known that the Stirling's formula%
\begin{equation}
n!\thicksim \sqrt{2\pi n}\left( \frac{n}{e}\right) ^{n}  \label{S}
\end{equation}%
for $n\in \mathbb{N}$ has various applications in probability theory,
statistical physics, number theory, combinatorics and other branches of
science. As a generalization of factorial function, the gamma function $%
\Gamma \left( x\right) =\int_{0}^{\infty }t^{x-1}e^{-t}dt$ for $x>0$ is no
exception. So many scholars pay attention to find various better
approximations for the factorial or gamma function, for example, Ramanujan 
\cite[P. 339]{Ramanujan-SB-1988}, Burnside \cite{Burnside-MM-46-1917},
Gosper \cite{Gosper-PNAS-75-1978}, Alzer \cite{Alzer-MC-66-1997}, \cite%
{Alzer-PRSE-139A-2009}, Windschitl (see \cite{Smith-2006} \cite{http/gamma}%
), Smith \cite{Smith-2006}, Batir \cite{Batir-P-27(1)-2008}, \cite%
{Batir-AM-91-2008}, Mortici \cite{Mortici-AM-93-2009-1}, \cite%
{Mortici-MMN-11(1)-2010}, \cite{Mortici-NA-56-2011}, \cite%
{Mortici-CMA-61-2011}, \cite{Mortici-MCM-57-2013}, \cite%
{Mortici0JMAA-402-2013} Nemes \cite{Nemes-AM-95-2010}, \cite%
{Nemes-TJM-9-2011}, Qi et al. \cite{Guo-JIPAM-9(1)-2008}, \cite%
{Qi-JCAM-268-2014}, Chen \cite{Chen-JNT-164-2016}, Yang et al. \cite%
{Yang-AMC-270-2015}, \cite{Yang-tian-jia}, \cite{Yang-JMAA-441-2016}, Lu et
al. \cite{Lu-RJ-35(1)-2014}, \cite{Lu-AMC-253-2015}.

As an asymptotic expansion of Stirling's formula (\ref{S}), one has the
Stirling's series for gamma function \cite[p. 257, Eq. (6.1.40)]%
{Abramowttz-1972} 
\begin{equation}
\Gamma \left( x+1\right) \thicksim \sqrt{2\pi x}\left( \frac{x}{e}\right)
^{x}\exp \left( \sum_{n=1}^{\infty }\frac{B_{2n}}{2n\left( 2n-1\right)
x^{2n-1}}\right)  \label{Stirling}
\end{equation}%
as $x\rightarrow \infty $, where $B_{2n}$ for $n\in \mathbb{N\cup }\mathbb{\{%
}0\mathbb{\}}$ is the Bernoulli number. It was proved in \cite[Theorem 8]%
{Alzer-MC-66-1997} by Alzer (see also \cite[Theorem 2]%
{Koumandos-JMAA-324-2006}) that for given integer $n\in \mathbb{N}$, the
function%
\begin{equation*}
F_{n}\left( x\right) =\ln \Gamma \left( x+1\right) -\left( x+\frac{1}{2}%
\right) \ln x+x-\frac{1}{2}\ln \left( 2\pi \right) -\sum_{k=1}^{n}\frac{%
B_{2k}}{2k\left( 2k-1\right) x^{2k-1}}
\end{equation*}%
is strictly completely monotonic on $\left( 0,\infty \right) $ if $n$ is
even, and so is $-F_{n}\left( x\right) $ if $n$ is odd. It thus follows that
the double inequality%
\begin{equation}
\exp \left( \sum_{k=1}^{2n}\frac{B_{2k}}{2k\left( 2k-1\right) x^{2k-1}}%
\right) <\frac{\Gamma \left( x+1\right) }{\sqrt{2\pi x}\left( x/e\right) ^{x}%
}<\exp \left( \sum_{k=1}^{2n-1}\frac{B_{2k}}{2k\left( 2k-1\right) x^{2k-1}}%
\right)  \label{g><}
\end{equation}%
holds for all $x>0$.

Another asymptotic expansion is the Laplace series (see \cite[p. 257, Eq.
(6.1.37)]{Abramowttz-1972}) 
\begin{equation}
\Gamma \left( x+1\right) \thicksim \sqrt{2\pi x}\left( \frac{x}{e}\right)
^{x}\left( 1+\frac{1}{12x}+\frac{1}{288x^{2}}-\frac{139}{51840x^{3}}-\frac{%
571}{2488320x^{4}}+\cdot \cdot \cdot \right)  \label{Laplace}
\end{equation}%
as $x\rightarrow \infty $. More asymptotic expansion developed by some
closed approximation formulas for gamma function can be found in \cite%
{Shi-JCAM-195-2006}, \cite{Feng-NA-64-2013}, \cite{Chen-AML-25-2012}, \cite%
{Chen-JCA-2-2013}, \cite{Chen-NA-64-2013}, \cite{Lu-JNT-136-2014}, \cite%
{Lu-JNT-140-2014}, \cite{Hirschhorn-RJ-34-2014}, \cite{Chen-AMC-245-2014}, 
\cite{Chen-AMC-261-2015}, \cite{Chen-JNT-149-2015}, \cite{Chen-AMC-250-2015}%
, \cite{Lin-JMI-9-2015}, \cite{Mortici-RJ-38-2015}, \cite{Xu-JNT-169-2016}
and the references cited therein.

Now let us focus on the Windschitl's approximation formula given by%
\begin{equation}
\Gamma \left( x+1\right) \thicksim W_{0}\left( x\right) =\sqrt{2\pi x}\left( 
\dfrac{x}{e}\right) ^{x}\left( x\sinh \frac{1}{x}\right) ^{x/2}\text{, as }%
x\rightarrow \infty .  \label{W0}
\end{equation}%
As shown in \cite[Eq. (3.8)]{Chen-JNT-164-2016}, the rate of Windschitl's
approximation $W_{0}\left( x\right) $ converging to $\Gamma \left(
x+1\right) $ is like $x^{-5}$ as $x\rightarrow \infty $, and is like $x^{-7%
\text{ }}$if replacing $W_{0}\left( x\right) $ with%
\begin{equation}
W_{1}\left( x\right) =\sqrt{2\pi x}\left( \dfrac{x}{e}\right) ^{x}\left(
x\sinh \frac{1}{x}\left[ +\frac{1}{810x^{6}}\right] \right) ^{x/2},
\label{W1}
\end{equation}%
by an easy check. These show that $W_{0}\left( x\right) $ and $W_{1}\left(
x\right) $ are excellent approximations for gamma function. Recently, Lu,
Song and Ma \cite{Lu-JNT-140-2014} extended Windschitl's formula to an
asymptotic expansion that%
\begin{equation}
\Gamma \left( n+1\right) \thicksim \sqrt{2\pi n}\left( \dfrac{n}{e}\right)
^{n}\left[ n\sinh \left( \frac{1}{n}+\frac{a_{7}}{n^{7}}+\frac{a_{9}}{n^{9}}+%
\frac{a_{11}}{n^{11}}+\cdot \cdot \cdot \right) \right] ^{n/2}
\label{g-ae-Lu}
\end{equation}%
as $n\rightarrow \infty $ with $%
a_{7}=1/810,a_{9}=-67/42525,a_{11}=19/8505,...$. An explicit formula for
determining the coefficients of $n^{-k}$ ($n\in \mathbb{N}$) was given in 
\cite[Theorem 1]{Chen-AMC-245-2014} by Chen. Other two asymptotic expansions%
\begin{eqnarray}
\Gamma \left( x+1\right) &\thicksim &\sqrt{2\pi x}\left( \dfrac{x}{e}\right)
^{x}\left( x\sinh \frac{1}{x}\right) ^{x/2+\sum_{j=0}^{\infty }r_{j}x^{-j}}%
\text{, }  \label{g-ae-Ch} \\
\Gamma \left( x+1\right) &\thicksim &\sqrt{2\pi x}\left( \dfrac{x}{e}\right)
^{x}\left( x\sinh \frac{1}{x}+\sum_{n=3}^{\infty }\frac{d_{n}}{x^{2n}}%
\right) ^{x/2},  \label{g-ae-W1}
\end{eqnarray}%
\ as $x\rightarrow \infty $ were presented in the papers \cite[Theorem 2]%
{Chen-AMC-245-2014}, \cite{Chen-AMC-250-2015}, respectively.

Inspired by the asymptotic expansions (\ref{g-ae-Lu}), (\ref{g-ae-Ch}), (\ref%
{g-ae-W1}) and Windschitl's approximation formula (\ref{W1}), the first aim
of this paper is to further present the following two asymptotic expansions
related to Windschitl's one (\ref{W0}): as $x\rightarrow \infty $,%
\begin{eqnarray}
\Gamma \left( x+1\right) &\thicksim &\sqrt{2\pi x}\left( \dfrac{x}{e}\right)
^{x}\left( x\sinh \frac{1}{x}\right) ^{x/2}\exp \left( \sum_{n=3}^{\infty }%
\frac{a_{n}}{x^{2n-1}}\right) ,  \label{g-ae-W0} \\
\Gamma \left( x+1\right) &\thicksim &\sqrt{2\pi x}\left( \dfrac{x}{e}\right)
^{x}\left( x\sinh \frac{1}{x}\right) ^{x/2}\left( 1+\sum_{n=1}^{\infty }%
\frac{b_{n}}{x^{n}}\right) ,  \label{g-ae-W0*}
\end{eqnarray}%
and give a more explicit coefficients formula in Chen's asymptotic expansion
(\ref{g-ae-Ch}). It is worth pointing out that those coefficients in (\ref%
{g-ae-W0}) have a closed-form expression, which is due to a little known
power series expansion of $\ln \left( t^{-1}\sinh t\right) $ (Lemma \ref%
{L-lnsh/t-ae}). We also give an estimate of remainder in asymptotic
expansion (\ref{g-ae-W0}). These results are presented in Section 2.

The second aim is to give some closed approximation formulas for gamma
function generated by truncating five asymptotic series just mentioned, and
compare accuracy of them by numeric computations and some inequalities.
These results are listed in Section 3.

\section{Asymptotic expansions and an estimate of remainder}

To obtain the explicit coefficients formulas in asymptotic expansions (\ref%
{g-ae-W0}), (\ref{g-ae-W0*}) and (\ref{g-ae-W1}), and to estimate the
remainder in asymptotic expansions (\ref{g-ae-W0}), we first give a lemma.

\begin{lemma}
\label{L-lnsh/t-ae}For $\left\vert t\right\vert <\pi $, we have%
\begin{equation}
\ln \frac{\sinh t}{t}=\sum_{n=1}^{\infty }\frac{2^{2n}B_{2n}}{2n\left(
2n\right) !}t^{2n}.  \label{lnsht/t-s}
\end{equation}%
Moreover, for $n\in \mathbb{N}$, the double inequality%
\begin{equation}
\sum_{k=1}^{2n}\frac{2^{2k}B_{2k}}{2k\left( 2k\right) !}t^{2k}<\ln \frac{%
\sinh t}{t}<\sum_{k=1}^{2n-1}\frac{2^{2k}B_{2k}}{2k\left( 2k\right) !}t^{2k}
\label{lnsht/t-s<>}
\end{equation}%
holds for all $t>0$.
\end{lemma}

\begin{proof}
It was listed in \cite[p. 85, Eq. (4.5.64), (4.5.65), (4.5.67)]%
{Abramowttz-1972} that%
\begin{equation*}
\coth t=\sum_{n=0}^{\infty }\frac{2^{2n}B_{2n}}{\left( 2n\right) !}t^{2n-1}\
\ \ \left\vert t\right\vert <\pi .
\end{equation*}%
Then we obtain that for $\left\vert t\right\vert <\pi $,%
\begin{equation*}
\ln \frac{\sinh t}{t}=\int_{0}^{t}\left( \coth x-\frac{1}{x}\right)
dx=\int_{0}^{t}\left( \sum_{n=0}^{\infty }\frac{2^{2n}B_{2n}}{\left(
2n\right) !}t^{2n-1}-\frac{1}{x}\right) dx=\sum_{n=1}^{\infty }\frac{%
2^{2n}B_{2n}}{2n\left( 2n\right) !}t^{2n}.
\end{equation*}%
While the double inquality (\ref{lnsht/t-s<>}) was proved in \cite[Corollary
1]{Yang-JMAA-441-2016}, which completes the proof.
\end{proof}

\begin{theorem}
\label{MT-W0}As $x\rightarrow \infty $, the asymptotic expansion%
\begin{equation*}
\Gamma \left( x+1\right) \thicksim \sqrt{2\pi x}\left( \dfrac{x}{e}\right)
^{x}\left( x\sinh \frac{1}{x}\right) ^{x/2}\exp \left( \sum_{n=3}^{\infty }%
\frac{a_{n}}{x^{2n-1}}\right) 
\end{equation*}%
\begin{equation*}
=\sqrt{2\pi x}\left( \dfrac{x}{e}\right) ^{x}\left( x\sinh \frac{1}{x}%
\right) ^{x/2}\exp \left[ \frac{1}{1620x^{5}}-\frac{11}{18\,900x^{7}}+\frac{%
143}{170\,100x^{9}}-\frac{2260\,261}{1178\,793\,000x^{11}}+\cdot \cdot \cdot %
\right] 
\end{equation*}%
holds with%
\begin{equation}
a_{n}=\frac{2n\left( 2n-2\right) !-2^{2n-1}}{2n\left( 2n\right) !}B_{2n},
\label{an}
\end{equation}%
where $B_{2n}$ is the Bernoulli number.
\end{theorem}

\begin{proof}
By the asymptotic expansion (\ref{Stirling}) and Lemma \ref{L-lnsh/t-ae} we
have that as $x\rightarrow \infty $,%
\begin{eqnarray*}
\ln \Gamma \left( x+1\right) -\ln \sqrt{2\pi }-\left( x+\frac{1}{2}\right)
\ln x+x &\thicksim &\sum_{n=1}^{\infty }\frac{a_{n}^{\prime }}{x^{2n-1}}, \\
\frac{x}{2}\ln \left( x\sinh \frac{1}{x}\right) &=&\frac{1}{2}%
\sum_{n=1}^{\infty }\frac{a_{n}^{\prime \prime }}{x^{2n-1}},
\end{eqnarray*}%
where%
\begin{equation*}
a_{n}^{\prime }=\frac{B_{2n}}{2n\left( 2n-1\right) }\text{ \ and \ }%
a_{n}^{\prime \prime }=\frac{2^{2n}B_{2n}}{2n\left( 2n\right) !}.
\end{equation*}

Let%
\begin{equation*}
\Gamma \left( x+1\right) \thicksim \sqrt{2\pi x}\left( \dfrac{x}{e}\right)
^{x}\left( x\sinh \frac{1}{x}\right) ^{x/2}\exp \left( w_{0}\left( x\right)
\right) \text{ as }x\rightarrow \infty .
\end{equation*}%
Then we have that as $x\rightarrow \infty $,%
\begin{eqnarray*}
w_{0}\left( x\right) &=&\left( \ln \Gamma \left( x+1\right) -\ln \sqrt{2\pi }%
-\left( x+\frac{1}{2}\right) \ln x+x\right) -\frac{x}{2}\ln \left( x\sinh 
\frac{1}{x}\right) \\
&=&\sum_{n=1}^{\infty }\frac{a_{n}^{\prime }}{x^{2n-1}}-\frac{1}{2}%
\sum_{n=1}^{\infty }\frac{a_{n}^{\prime \prime }}{x^{2n-1}}%
=\sum_{n=1}^{\infty }\frac{a_{n}^{\prime }-a_{n}^{\prime \prime }/2}{x^{2n-1}%
}=\sum_{n=1}^{\infty }\frac{a_{n}}{x^{2n-1}}.
\end{eqnarray*}%
An easy computation yields $a_{1}=a_{2}=0$ and%
\begin{equation*}
a_{3}=\frac{1}{1620}\text{, \ }a_{4}=-\frac{11}{18\,900}\text{, \ }a_{5}=%
\frac{143}{170\,100}\text{, \ }a_{6}=-\frac{2260\,261}{1178\,793\,000}
\end{equation*}%
which completes the proof.
\end{proof}

The following theorem offers an estimate of remainder in asymptotic
expansion (\ref{g-ae-W0}).

\begin{theorem}
\label{MT-W0-Rn}For $n\in \mathbb{N}$ with $n\geq 4$, let%
\begin{equation*}
\Gamma \left( x+1\right) =\sqrt{2\pi x}\left( \dfrac{x}{e}\right) ^{x}\left(
x\sinh \frac{1}{x}\right) ^{x/2}\exp \left( \sum_{k=3}^{n-1}\frac{a_{k}}{%
x^{2k-1}}+R_{n}\left( x\right) \right) ,
\end{equation*}%
where $a_{k}$ is given by (\ref{an}). Then we have%
\begin{equation*}
\left\vert R_{n}\left( x\right) \right\vert \leq \frac{\left\vert
B_{2n}\right\vert }{2n\left( 2n-1\right) }\frac{1}{x^{2n-1}}
\end{equation*}%
for all $x>0$.
\end{theorem}

\begin{proof}
We have%
\begin{equation*}
R_{n}\left( x\right) =\ln \Gamma \left( x+1\right) -\ln \sqrt{2\pi }-\left(
x+\frac{1}{2}\right) \ln x+x-\frac{x}{2}\ln \left( x\sinh \frac{1}{x}\right)
-\sum_{k=3}^{n-1}\frac{a_{k}}{x^{2k-1}}.
\end{equation*}

If $n=2m+1$ for $m\geq 2$ then by inequalities (\ref{g><}) and (\ref%
{lnsht/t-s<>}) we have%
\begin{eqnarray*}
R_{2m}\left( x\right)  &=&\left( \ln \Gamma \left( x+1\right) -\ln \sqrt{%
2\pi }-\left( x+\frac{1}{2}\right) \ln x+x\right) -\frac{x}{2}\ln \left(
x\sinh \frac{1}{x}\right) -\sum_{k=3}^{2m}\frac{a_{k}}{x^{2k-1}} \\
&>&\sum_{k=1}^{2m}\dfrac{a_{k}^{\prime }}{x^{2k-1}}-\frac{1}{2}%
\sum_{k=1}^{2m+1}\frac{a_{k}^{\prime \prime }}{x^{2k-1}}-\sum_{k=3}^{2m}%
\frac{a_{m}}{x^{2k-1}}=-\frac{1}{2}\frac{a_{2m+1}^{\prime \prime }}{x^{4m+1}}%
,
\end{eqnarray*}%
\begin{equation*}
R_{2m}\left( x\right) <\sum_{k=1}^{2m+1}\dfrac{a_{k}^{\prime }}{x^{2k-1}}-%
\frac{1}{2}\sum_{k=1}^{2m}\frac{\ddot{a}_{k}}{x^{2k-1}}-\sum_{k=3}^{2m}\frac{%
a_{m}}{x^{2k-1}}=\frac{a_{2m+1}^{\prime }}{x^{4m+1}},
\end{equation*}%
where the last eqaulities in the above two inequalities hold due to $%
a_{k}^{\prime }-a_{k}^{\prime \prime }/2=a_{k}$. It then follows that%
\begin{equation*}
\left\vert R_{n}\left( x\right) \right\vert <\max \left( \left\vert -\frac{1%
}{2}\frac{a_{2m+1}^{\prime \prime }}{x^{4m+1}}\right\vert ,\left\vert \frac{%
a_{2m+1}^{\prime }}{x^{4m+1}}\right\vert \right) =\max \left( \frac{1}{2}%
\left\vert a_{n}^{\prime \prime }\right\vert ,\left\vert a_{n}^{\prime
}\right\vert \right) \frac{1}{x^{2n-1}}.
\end{equation*}%
By a similar verification, it is also true if $n=2m$ for $m\geq 2$.

Since%
\begin{eqnarray*}
\frac{\left\vert a_{n}^{\prime }\right\vert }{\left\vert a_{n}^{\prime
\prime }/2\right\vert } &=&\left\vert \frac{B_{2n}}{2n\left( 2n-1\right) }%
\right\vert \left/ \left\vert \frac{1}{2}\frac{2^{2n}B_{2n}}{2n\left(
2n\right) !}\right\vert \right. =2\frac{\left( 2n\right) !}{\left(
2n-1\right) 2^{2n}}:=a_{n}^{\prime \prime \prime }, \\
\frac{a_{n+1}^{\prime \prime \prime }}{a_{n}^{\prime \prime \prime }}-1 &=&%
\frac{1}{2}\left( 2n+3\right) \left( n-1\right) >0,
\end{eqnarray*}%
it is derived that $a_{n}^{\prime \prime \prime }>a_{4}^{\prime \prime
\prime }=45$ for $n\geq 4$, so we obtain%
\begin{equation*}
\max \left( \frac{1}{2}\left\vert a_{n}^{\prime \prime }\right\vert
,\left\vert a_{n}^{\prime }\right\vert \right) =\left\vert a_{n}^{\prime
}\right\vert =\frac{\left\vert B_{2n}\right\vert }{2n\left( 2n-1\right) },
\end{equation*}%
which completes the proof.
\end{proof}

\begin{remark}
Since $B_{2n+1}=0$ for $n\in \mathbb{N}$, the asymptotic series $w_{0}\left(
x\right) $ can also be written as%
\begin{equation*}
w_{0}\left( x\right) =\sum_{n=3}^{\infty }\frac{2n\left( 2n-2\right)
!-2^{2n-1}}{2n\left( 2n\right) !}\frac{B_{2n}}{x^{2n-1}}=\sum_{n=1}^{\infty }%
\frac{\left( n+1\right) \left( n-1\right) !-2^{n}}{\left( n+1\right) \left(
n+1\right) !}\frac{B_{n+1}}{x^{n}}:=\sum_{n=1}^{\infty }\frac{a_{n}^{\ast }}{%
x^{n}},
\end{equation*}%
where%
\begin{equation}
a_{n}^{\ast }=\frac{\left( n+1\right) \left( n-1\right) !-2^{n}}{\left(
n+1\right) \left( n+1\right) !}B_{n+1}.  \label{an*}
\end{equation}
\end{remark}

Now we establish the second Windschitl type asymptotic series for gamma
function.

\begin{theorem}
\label{MT-W0*}As $x\rightarrow \infty $, the asymptotic expansion (\ref%
{g-ae-W1})%
\begin{equation*}
\Gamma \left( x+1\right) \thicksim \sqrt{2\pi x}\left( \dfrac{x}{e}\right)
^{x}\left( x\sinh \frac{1}{x}\right) ^{x/2}\left( 1+\sum_{n=1}^{\infty }%
\frac{b_{n}}{x^{n}}\right)
\end{equation*}%
\begin{equation*}
=\sqrt{2\pi x}\left( \dfrac{x}{e}\right) ^{x}\left( x\sinh \frac{1}{x}%
\right) ^{x/2}\left( 1+\frac{1}{1620x^{5}}-\frac{11}{18\,900x^{7}}+\frac{143%
}{170\,100x^{9}}+\frac{1}{5248\,800x^{10}}+\cdot \cdot \cdot \right)
\end{equation*}%
holds with $b_{0}=1$, $b_{1}=b_{2}=b_{3}=b_{4}=0$ and for $n\geq 5$,%
\begin{equation}
b_{n}=\frac{1}{n}\sum_{k=1}^{n}\left( \frac{1}{k+1}-\frac{2^{k}}{\left(
k+1\right) ^{2}\left( k-1\right) !}\right) B_{k+1}b_{n-k}.  \label{bn}
\end{equation}
\end{theorem}

\begin{proof}
It was proved in \cite[Lemma 3]{Chen-JCA-2-2013} that as $x\rightarrow
\infty $,%
\begin{equation*}
\exp \left( \sum_{n=1}^{\infty }a_{n}x^{-n}\right) \thicksim
\sum_{n=0}^{\infty }b_{n}x^{-n}
\end{equation*}%
with $b_{0}=1$ and%
\begin{equation}
b_{n}=\frac{1}{n}\sum_{k=1}^{n}ka_{k}b_{n-k}\text{ \ for }n\geq 1.
\label{bn=}
\end{equation}

Substituting $a_{k}^{\ast }$ given in (\ref{an*}) into (\ref{bn=}) gives
recurrence formula (\ref{bn}).

An easy verification shows that $b_{n}=0$ for $1\leq n\leq 4$, $%
b_{6}=b_{8}=0 $ and%
\begin{equation*}
b_{5}=\frac{1}{1620}\text{, }b_{7}=-\frac{11}{18\,900}\text{, }b_{9}=\frac{%
143}{170\,100}\text{, }b_{10}=\frac{1}{5248\,800}\text{,}
\end{equation*}%
which completes the proof.
\end{proof}

The following theorem improves Chen's result \cite[Theorem 2]%
{Chen-AMC-245-2014}.

\begin{theorem}
As $x\rightarrow \infty $, the asymptotic expansion%
\begin{equation}
\Gamma \left( x+1\right) \thicksim \sqrt{2\pi x}\left( \dfrac{x}{e}\right)
^{x}\left( x\sinh \frac{1}{x}\right) ^{\left( x/2\right) \left(
1+\sum_{n=2}^{\infty }c_{n}x^{-2n}\right) }  \label{g-ae-C}
\end{equation}%
\begin{equation*}
=\sqrt{2\pi x}\left( \dfrac{x}{e}\right) ^{x}\left( x\sinh \frac{1}{x}%
\right) ^{\frac{x}{2}\left( 1+\frac{1}{135x^{4}}-\frac{191}{28\,350x^{6}}+%
\frac{25\,127}{2551\,500x^{8}}-\frac{19\,084\,273}{841\,995\,000x^{10}}%
+\cdot \cdot \cdot \right) }
\end{equation*}%
holds with $c_{0}=1$, $c_{1}=0$ and for $n\geq 2$,%
\begin{equation*}
c_{n}=\frac{6B_{2n+2}}{\left( n+1\right) \left( 2n+1\right) }-6\sum_{k=1}^{n}%
\frac{2^{2k+2}B_{2k+2}}{2\left( k+1\right) \left( 2k+2\right) !}c_{n-k}.
\end{equation*}
\end{theorem}

\begin{proof}
The asymptotic expansion (\ref{g-ae-Ch}) can be written as%
\begin{equation*}
\ln \Gamma \left( x+1\right) -\ln \sqrt{2\pi x}-x\ln x+x\thicksim \left( 
\frac{x}{2}+\sum_{j=0}^{\infty }\frac{r_{j}}{x^{j}}\right) \ln \left( x\sinh 
\frac{1}{x}\right) ,
\end{equation*}%
which, by (\ref{Stirling}) and (\ref{lnsht/t-s}), is equivalent to%
\begin{equation*}
\sum_{n=1}^{\infty }\frac{B_{2n}}{2n\left( 2n-1\right) }\frac{1}{x^{2n-1}}%
\thicksim \left( \frac{x}{2}+\sum_{j=0}^{\infty }\frac{r_{j}}{x^{j}}\right)
\left( \sum_{n=1}^{\infty }\frac{2^{2n}B_{2n}}{2n\left( 2n\right) !}\frac{1}{%
x^{2n}}\right) .
\end{equation*}%
Since the left hand side one and the second factor of the right hand side
one are odd and even, respectively, the asymptotic expansion $%
x/2+\sum_{j=0}^{\infty }r_{j}x^{-j}$ has to be odd, and so $r_{2n}=0$ for $%
n\in \mathbb{N\cup \{}0\}$. Then the asymptotic expansion (\ref{g-ae-Ch})
has the form of (\ref{g-ae-C}), which is equivalent to%
\begin{equation*}
\sum_{n=1}^{\infty }\frac{B_{2n}}{2n\left( 2n-1\right) }\frac{1}{x^{2n-1}}%
\thicksim \frac{x}{2}\left( \sum_{n=0}^{\infty }\frac{c_{n}}{x^{2n}}\right)
\left( \sum_{n=1}^{\infty }\frac{2^{2n}B_{2n}}{2n\left( 2n\right) !}\frac{1}{%
x^{2n}}\right) .
\end{equation*}%
It can be written as%
\begin{equation*}
\sum_{n=0}^{\infty }\frac{B_{2n+2}}{\left( n+1\right) \left( 2n+1\right) }%
\frac{1}{x^{2n}}\thicksim \left( \sum_{n=0}^{\infty }\frac{c_{n}}{x^{2n}}%
\right) \left( \sum_{n=1}^{\infty }\frac{2^{2n+2}B_{2n+2}}{2\left(
n+1\right) \left( 2n+2\right) !}\frac{1}{x^{2n}}\right) .
\end{equation*}

Comprising coefficients of $x^{-2n}$ gives%
\begin{equation*}
\frac{B_{2n+2}}{\left( n+1\right) \left( 2n+1\right) }=\sum_{k=0}^{n}\frac{%
2^{2k+2}B_{2k+2}}{2\left( k+1\right) \left( 2k+2\right) !}c_{n-k},
\end{equation*}%
which yields $c_{0}=1$ and for $n\geq 1$,%
\begin{equation*}
c_{n}=\frac{6B_{2n+2}}{\left( n+1\right) \left( 2n+1\right) }-6\sum_{k=1}^{n}%
\frac{2^{2k+2}B_{2k+2}}{2\left( k+1\right) \left( 2k+2\right) !}c_{n-k}.
\end{equation*}%
A straightforward computation leads to 
\begin{equation*}
c_{1}=0\text{, }c_{2}=\frac{1}{135}\text{, }c_{3}=-\frac{191}{28\,350}\text{%
, }c_{4}=\frac{25\,127}{2551\,500}\text{, }c_{5}=-\frac{19\,084\,273}{%
841\,995\,000},
\end{equation*}%
which ends the proof.
\end{proof}

\begin{remark}
Chen's recurrence formula of coefficients $r_{j}$ given in \cite[Theorem 2]%
{Chen-AMC-245-2014} is somewhat complicated, since he was unaware of the
power series (\ref{lnsht/t-s}).
\end{remark}

\section{Numeric comparisons and inequalities}

If the series in (\ref{g-ae-W0}), (\ref{g-ae-W0*}), (\ref{g-ae-W1}) (\ref%
{g-ae-C}) are truncated at $n=3,5,3$, $2$, respectively, then we obtain four
Windschitl type approximation formulas:

\begin{eqnarray}
\Gamma \left( x+1\right) &\thicksim &\sqrt{2\pi x}\left( \dfrac{x}{e}\right)
^{x}\left( x\sinh \frac{1}{x}\right) ^{x/2}\exp \left( \frac{1}{1620x^{5}}%
\right) :=W_{01}\left( x\right) ,  \label{W01} \\
\Gamma \left( x+1\right) &\thicksim &\sqrt{2\pi x}\left( \dfrac{x}{e}\right)
^{x}\left( x\sinh \frac{1}{x}\right) ^{x/2}\left( 1+\frac{1}{1620x^{5}}%
\right) :=W_{01}^{\ast }\left( x\right) ,  \label{W01*} \\
\Gamma \left( x+1\right) &\thicksim &\sqrt{2\pi x}\left( \dfrac{x}{e}\right)
^{x}\left( x\sinh \frac{1}{x}+\frac{1}{810x^{6}}\right) ^{x/2}=W_{1}\left(
x\right) ,  \notag \\
\Gamma \left( x+1\right) &\thicksim &\sqrt{2\pi x}\left( \dfrac{x}{e}\right)
^{x}\left( x\sinh \frac{1}{x}\right) ^{\frac{x}{2}\left( 1+\frac{1}{135x^{4}}%
\right) }:=W_{c1}\left( x\right) ,  \label{Wc1}
\end{eqnarray}%
as $x\rightarrow \infty $. Also, we denote Lu et al.'s one \cite[Theorem 1.8]%
{Lu-JNT-140-2014} by%
\begin{equation}
W_{l1}\left( x\right) =\sqrt{2\pi x}\left( \dfrac{x}{e}\right) ^{x}\left(
x\sinh \left( \frac{1}{x}+\frac{1}{810x^{7}}\right) \right) ^{x/2}.
\label{Wl1}
\end{equation}%
In this section, we aim to compare the five closed approximation formulas
listed above.

We easily obtain%
\begin{eqnarray*}
\lim_{x\rightarrow \infty }\frac{\ln \Gamma \left( x+1\right) -\ln
W_{1}\left( x\right) }{x^{-7}} &=&-\frac{163}{340\,200}, \\
\lim_{x\rightarrow \infty }\frac{\ln \Gamma \left( x+1\right) -\ln
W_{c1}\left( x\right) }{x^{-7}} &=&-\frac{191}{340\,200},
\end{eqnarray*}%
\begin{eqnarray*}
\lim_{x\rightarrow \infty }\frac{\ln \Gamma \left( x+1\right) -\ln
W_{01}\left( x\right) }{x^{-7}} &=&\lim_{x\rightarrow \infty }\frac{\ln
\Gamma \left( x+1\right) -\ln W_{01}^{\ast }\left( x\right) }{x^{-7}}=-\frac{%
198}{340\,200}, \\
\lim_{x\rightarrow \infty }\frac{\ln \Gamma \left( x+1\right) -\ln
W_{c1}\left( x\right) }{x^{-7}} &=&-\frac{268}{340\,200}.
\end{eqnarray*}%
These show that the rates of the five approximation ones converging to $%
\Gamma \left( x+1\right) $ are all like $x^{-7}$ as $x\rightarrow \infty $,
and $W_{1}\left( x\right) $ may be the best one among all five approximation
ones, which can also be seen from the following Table 1.

\begin{equation*}
\text{Table 1: Comparisons among }W_{1}\left( x\right) ,W_{c1}\left(
x\right) ,W_{01}\left( x\right) ,W_{l1}\left( x\right)
\end{equation*}%
\begin{equation*}
\begin{tabular}{|l|l|l|l|l|}
\hline
$x$ & $\left\vert \frac{W_{1}\left( x\right) -\Gamma \left( x+1\right) }{%
\Gamma \left( x+1\right) }\right\vert $ & $\left\vert \frac{W_{c1}\left(
x\right) -\Gamma \left( x+1\right) }{\Gamma \left( x+1\right) }\right\vert $
& $\left\vert \frac{W_{01}\left( x\right) -\Gamma \left( x+1\right) }{\Gamma
\left( x+1\right) }\right\vert $ & $\left\vert \frac{W_{l1}\left( x\right)
-\Gamma \left( x+1\right) }{\Gamma \left( x+1\right) }\right\vert $ \\ \hline
$1$ & $1.832\times 10^{-4}$ & $2.562\times 10^{-4}$ & $2.754\times 10^{-4}$
& $4.686\times 10^{-4}$ \\ \hline
$2$ & $2.668\times 10^{-6}$ & $3.292\times 10^{-6}$ & $3.449\times 10^{-6}$
& $5.030\times 10^{-6}$ \\ \hline
$5$ & $5.743\times 10^{-9}$ & $6.791\times 10^{-9}$ & $7.054\times 10^{-9}$
& $9.681\times 10^{-9}$ \\ \hline
$10$ & $4.710\times 10^{-11}$ & $5.532\times 10^{-11}$ & $5.738\times
10^{-11}$ & $7.794\times 10^{-11}$ \\ \hline
$20$ & $3.727\times 10^{-13}$ & $4.370\times 10^{-13}$ & $4.531\times
10^{-13}$ & $6.138\times 10^{-13}$ \\ \hline
$50$ & $6.129\times 10^{-16}$ & $7.182\times 10^{-16}$ & $7.446\times
10^{-16}$ & $1.008\times 10^{-15}$ \\ \hline
$100$ & $4.791\times 10^{-18}$ & $5.614\times 10^{-18}$ & $5.819\times
10^{-18}$ & $7.877\times 10^{-18}$ \\ \hline
\end{tabular}%
\end{equation*}

More precisely, we have the following theorem.

\begin{theorem}
\label{T-f1}(i) The function%
\begin{equation*}
f_{1}\left( x\right) =\ln \Gamma \left( x+1\right) -\ln \sqrt{2\pi }-\left(
x+\frac{1}{2}\right) \ln x+x-\frac{x}{2}\ln \left( x\sinh \frac{1}{x}+\frac{1%
}{810x^{6}}\right)
\end{equation*}%
is strictly increasing and concave on $[1,\infty )$.

(ii) For $x\geq 1$, we have%
\begin{equation}
\begin{array}{l}
\beta _{0}\left( x\sinh \dfrac{1}{x}+\dfrac{1}{810x^{6}}\right) ^{x/2}<%
\dfrac{\Gamma \left( x+1\right) }{\sqrt{2\pi x}\left( x/e\right) ^{x}}%
<\left( x\sinh \dfrac{1}{x}+\dfrac{1}{810x^{6}}\right) ^{x/2}\bigskip \\ 
<\left( x\sinh \dfrac{1}{x}\right) ^{\frac{x}{2}\left( 1+\frac{1}{135x^{4}}%
\right) }<\left( x\sinh \dfrac{1}{x}\right) ^{x/2}\left( 1+\dfrac{1}{%
1620x^{5}}\right) \bigskip \\ 
<\left( x\sinh \dfrac{1}{x}\right) ^{x/2}\exp \left( \dfrac{1}{1620x^{5}}%
\right) <\sqrt{2\pi x}\left( \dfrac{x}{e}\right) ^{x}\left( x\sinh \left( 
\dfrac{1}{x}+\dfrac{1}{810x^{7}}\right) \right) ^{x/2}%
\end{array}
\label{MI}
\end{equation}%
with the best constant%
\begin{equation*}
\beta _{0}=\frac{e}{\sqrt{2\pi \sinh 1+\pi /405}}\approx 0.999\,81.
\end{equation*}
\end{theorem}

To prove this theorem, we need three lemmas.

\begin{lemma}
\label{L-2.1}The inequalities%
\begin{equation}
\psi ^{\prime }\left( x+\frac{1}{2}\right) <\frac{1}{x}\frac{x^{4}+\frac{67}{%
36}x^{2}+\frac{256}{945}}{x^{4}+\frac{35}{18}x^{2}+\frac{407}{1008}}
\label{dpsi<I}
\end{equation}%
hold for all $x>0$.
\end{lemma}

\begin{proof}
Let%
\begin{equation*}
g\left( x\right) =\psi\left( x+\frac{1}{2},1\right) -\frac{1}{x}\frac{x^{4}+%
\frac{67}{36}x^{2}+\frac{256}{945}}{x^{4}+\frac{35}{18}x^{2}+\frac{407}{1008}%
}.
\end{equation*}%
Then we have 
\begin{eqnarray*}
g\left( x+1\right) -g\left( x\right) &=&\psi\left( x+\frac{3}{2},1\right) -%
\frac{1}{x+1}\frac{\left( x+1\right) ^{4}+\frac{67}{36}\left( x+1\right)
^{2}+\frac{256}{945}}{\left( x+1\right) ^{4}+\frac{35}{18}\left( x+1\right)
^{2}+\frac{407}{1008}} \\
&&-\psi\left( x+\frac{1}{2},1\right) +\frac{1}{x}\frac{x^{4}+\frac{67}{36}%
x^{2}+\frac{256}{945}}{x^{4}+\frac{35}{18}x^{2}+\frac{407}{1008}}
\end{eqnarray*}%
\begin{equation*}
=\dfrac{921\,600}{x\left( x+1\right) \left( 2x+1\right) ^{2}\left(
1008x^{4}+1960x^{2}+407\right) \left(
1008x^{4}+4032x^{3}+8008x^{2}+7952x+3375\right) }>0.
\end{equation*}%
Hence, we conclude that%
\begin{equation*}
g\left( x\right) <g\left( x+1\right) <\cdot \cdot \cdot <\lim_{n\rightarrow
\infty }g\left( x+n\right) =0,
\end{equation*}%
which proves (\ref{dpsi<I}), and the proof is done.
\end{proof}

The second lemma offers a simple criterion to determine the sign of a\ class
of special polynomial on given interval contained in $\left( 0,\infty
\right) $ without using Descartes' Rule of Signs, which play an important
role in studying for certain special functions, see for example \cite%
{Yang-AAA-2014-702718}, \cite{Yang-JMI-12-2018}. A series version can be
found in \cite{Yang-arxiv-1705-05704}.

\begin{lemma}[{\protect\cite[Lemma 7]{Yang-AAA-2014-702718}}]
\label{L-pz}Let $n\in \mathbb{N}$ and $m\in \mathbb{N}\cup \{0\}$ with $n>m$
and let $P_{n}\left( t\right) $ be an $n$ degrees polynomial defined by 
\begin{equation}
P_{n}\left( t\right) =\sum_{i=m+1}^{n}a_{i}t^{i}-\sum_{i=0}^{m}a_{i}t^{i},
\label{2.4}
\end{equation}%
where $a_{n},a_{m}>0$, $a_{i}\geq 0$ for $0\leq i\leq n-1$ with $i\neq m$.
Then there is a unique number $t_{m+1}\in \left( 0,\infty \right) $ to
satisfy $P_{n}\left( t\right) =0$ such that $P_{n}\left( t\right) <0$ for $%
t\in \left( 0,t_{m+1}\right) $ and $P_{n}\left( t\right) >0$ for $t\in
\left( t_{m+1},\infty \right) $.

Consequently, for given $t_{0}>0$, if $P_{n}\left( t_{0}\right) >0$ then $%
P_{n}\left( t\right) >0$ for $t\in \left( t_{0},\infty \right) $ and if $%
P_{n}\left( t_{0}\right) <0$ then $P_{n}\left( t\right) <0$ for $t\in \left(
0,t_{0}\right) $.
\end{lemma}

\begin{lemma}
\label{L-Wk}Let $W_{01}\left( x\right) $, $W_{01}^{\ast }\left( x\right) $, $%
W_{1}\left( x\right) $, $W_{c1}\left( x\right) $ and $W_{l1}\left( x\right) $
be defined by (\ref{W01}), (\ref{W01*}), (\ref{W1}), (\ref{Wc1}) and (\ref%
{Wl1}), respectively. Then we have%
\begin{equation*}
W_{1}\left( x\right) <W_{c1}\left( x\right) <W_{01}^{\ast }\left( x\right)
<W_{01}\left( x\right) <W_{l1}\left( x\right)
\end{equation*}%
for all $x\geq 1$.
\end{lemma}

\begin{proof}
(i) The first inequality $W_{1}\left( x\right) <W_{c1}\left( x\right) $ is
equivalent to%
\begin{equation*}
h_{1}\left( t\right) =\ln \left( \frac{\sinh t}{t}+\frac{1}{810}t^{6}\right)
-\left( 1+\frac{1}{135}t^{4}\right) \ln \left( \frac{\sinh t}{t}\right) <0
\end{equation*}%
for $t=1/x\in (0,1]$. We have%
\begin{equation*}
\frac{d}{dy}\left( \ln \left( y+\frac{1}{810}t^{6}\right) -\left( 1+\frac{1}{%
135}t^{4}\right) \ln y\right) =-\frac{1}{135}t^{4}\frac{810y+135t^{2}+t^{6}}{%
y\left( 810y+t^{6}\right) }<0
\end{equation*}%
for $y>1$, which together with the inequality%
\begin{equation*}
y=\frac{\sinh t}{t}>1+\frac{1}{6}t^{2}
\end{equation*}%
for $t>0$ yields%
\begin{equation*}
h_{1}\left( t\right) <\ln \left( 1+\frac{1}{6}t^{2}+\frac{1}{810}%
t^{6}\right) -\left( 1+\frac{1}{135}t^{4}\right) \ln \left( 1+\frac{1}{6}%
t^{2}\right) :=h_{11}\left( t\right) .
\end{equation*}%
Differentiation leads us to%
\begin{eqnarray*}
\frac{135}{2t^{3}}h_{11}^{\prime }\left( t\right)  &=&-2\ln \left( \frac{1}{6%
}t^{2}+1\right) -t^{2}\frac{t^{6}-135t^{2}-1620}{\left( t^{2}+6\right)
\left( t^{6}+135t^{2}+810\right) }:=h_{12}\left( t\right) , \\
h_{12}^{\prime }\left( t\right)  &=&-4\frac{t^{3}\left(
t^{12}+9t^{10}+540t^{8}+8505t^{6}+47\,385t^{4}+328\,050t^{2}+1312\,200%
\right) }{\left( t^{2}+6\right) ^{2}\left( t^{6}+135t^{2}+810\right) ^{2}}<0
\end{eqnarray*}%
for $t>0$. Therefore, we obtain $h_{12}\left( t\right) <h_{12}\left(
0\right) =0$, and so $h_{11}\left( t\right) <h_{11}\left( 0\right) =0$,
which implies $h_{1}\left( t\right) <0$ for $t>0$.

(ii) The second one $W_{c1}\left( x\right) <W_{01}^{\ast }\left( x\right) $
is equivalent to%
\begin{equation*}
\frac{x}{2}\left( 1+\frac{1}{135x^{4}}\right) \ln \left( x\sinh \dfrac{1}{x}%
\right) <\frac{x}{2}\ln \left( x\sinh \dfrac{1}{x}\right) +\ln \left( 1+%
\dfrac{1}{1620x^{5}}\right) ,
\end{equation*}%
or equivalently, 
\begin{equation*}
h_{2}\left( t\right) =\frac{1}{270}t^{3}\ln \left( \frac{\sinh t}{t}\right)
-\ln \left( 1+\dfrac{1}{1620}t^{5}\right) <0
\end{equation*}%
for $t=1/x\in (0,1]$. Taking $n=2$ in the inequalities (\ref{lnsht/t-s<>})
gives%
\begin{equation*}
\ln \left( \frac{\sinh t}{t}\right) <\frac{1}{6}t^{2}-\frac{1}{180}t^{4}+%
\frac{1}{2835}t^{6},
\end{equation*}%
which is applied to the expression of $h_{2}\left( t\right) $:%
\begin{equation*}
h_{2}\left( t\right) <\frac{1}{270}t^{3}\left( \frac{1}{6}t^{2}-\frac{1}{180}%
t^{4}+\frac{1}{2835}t^{6}\right) -\ln \left( 1+\dfrac{1}{1620}t^{5}\right)
:=h_{21}\left( t\right) .
\end{equation*}%
Differentiation yields%
\begin{equation*}
h_{21}^{\prime }\left( t\right) =\frac{t^{6}}{340\,200}\frac{%
4t^{7}-49t^{5}+1050t^{3}+6480t^{2}-79\,380}{t^{5}+1620}<0
\end{equation*}%
for $t\in (0,1]$, which proves $h_{2}\left( t\right) <0$ for $t\in (0,1]$.

(iii) The third one $W_{01}^{\ast }\left( x\right) <W_{01}\left( x\right) $
is equivalent to%
\begin{equation*}
1+\frac{1}{1620x^{5}}<\exp \left( \frac{1}{1620x^{5}}\right) ,
\end{equation*}%
which follows by a simple inequality $1+y<e^{y}$ for $y\in \mathbb{R}$.

(iv) The fourth one $W_{01}\left( x\right) <W_{l1}\left( x\right) $ is
equivalent to%
\begin{equation*}
\frac{x}{2}\ln \left( x\sinh \left( \frac{1}{x}+\frac{1}{810x^{7}}\right)
\right) >\frac{x}{2}\ln \left( x\sinh \frac{1}{x}\right) +\frac{1}{1620x^{5}}%
,
\end{equation*}%
or equivalently,%
\begin{equation*}
h_{3}\left( t\right) =\ln \sinh \left( t+\frac{1}{810}t^{7}\right) -\ln
\sinh t-\frac{1}{810}t^{6}>0
\end{equation*}%
for $t=1/x>0$. Denote by $h_{30}\left( t\right) =\ln \sinh t$. Then by
Taylor formula we have%
\begin{eqnarray*}
h_{3}\left( t\right) &=&h_{30}\left( t+\frac{1}{810}t^{7}\right)
-h_{30}\left( t\right) -\frac{1}{810}t^{6} \\
&=&\frac{t^{7}}{810}h_{30}^{\prime }\left( t\right) +\frac{1}{2!}\frac{t^{14}%
}{810^{2}}h_{30}^{\prime \prime }\left( t\right) +\frac{1}{3!}\frac{t^{21}}{%
810^{3}}h_{30}^{\prime \prime \prime }\left( \xi \right) -\frac{1}{810}t^{6},
\end{eqnarray*}%
where $t<\xi <t+t^{7}/810$. Since $h_{30}^{\prime \prime \prime }\left(
t\right) =2\left( \cosh t\right) /\sinh ^{3}t>0$, we get 
\begin{equation*}
h_{3}\left( t\right) >\frac{1}{810}t^{7}\frac{\cosh t}{\sinh t}-\frac{t^{14}%
}{2\times 810^{2}}\frac{1}{\sinh ^{2}t}-\frac{1}{810}t^{6}:=\frac{%
t^{6}\times h_{31}\left( t\right) }{2\times 810^{2}\sinh ^{2}t},
\end{equation*}%
where%
\begin{equation*}
h_{31}\left( t\right) =810t\sinh 2t-810\cosh 2t+810-t^{8}.
\end{equation*}%
Due to%
\begin{equation*}
h_{31}\left( t\right) =540t^{4}+144t^{6}+\frac{101}{7}t^{8}+810\sum_{n=5}^{%
\infty }\frac{\left( n-1\right) 2^{2n}}{\left( 2n\right) !}t^{2n}>0,
\end{equation*}%
we conclude that $h_{3}\left( t\right) >0$ for $t>0$, which completes the
proof.
\end{proof}

We are now in a position to prove Theorem \ref{T-f1}.

\begin{proof}[Proof of Theorem \protect\ref{T-f1}]
(i) Differentiation yields%
\begin{eqnarray*}
f_{1}^{\prime }\left( x\right) &=&\psi \left( x+1\right) -\ln x-\frac{1}{2x}-%
\frac{1}{2}\ln \left( x\sinh \frac{1}{x}+\frac{1}{810x^{6}}\right) \\
&&+3\frac{135x^{6}\cosh \frac{1}{x}-135x^{7}\sinh \frac{1}{x}+1}{%
810x^{7}\sinh \frac{1}{x}+1},
\end{eqnarray*}%
\begin{equation*}
\begin{array}{l}
f_{1}^{\prime \prime }\left( x\right) =\psi ^{\prime }\left( x+1\right) -%
\dfrac{1}{x}+\dfrac{1}{2x^{2}}\bigskip \\ 
-\dfrac{3\left( 109\,350x^{14}\sinh ^{2}\frac{1}{x}+5940x^{7}\sinh \frac{1}{x%
}+135x^{5}\sinh \frac{1}{x}-1890x^{6}\cosh \frac{1}{x}-109\,350x^{12}-1%
\right) }{x\left( 810x^{7}\sinh \frac{1}{x}+1\right) ^{2}}.%
\end{array}%
\end{equation*}

Replacing $x$ by $x+1/2$ in inequality (\ref{dpsi<I}) yields%
\begin{equation*}
\psi ^{\prime }\left( x+1\right) <\frac{1}{30}\frac{3780x^{4}+7560x^{3}+12%
\,705x^{2}+8925x+3019}{\left( 2x+1\right) \left(
63x^{4}+126x^{3}+217x^{2}+154x+60\right) },
\end{equation*}%
for $x>-1/2$, and applying which then making a change of variable $x=1/t\in
(0,1]$ yield%
\begin{equation*}
f_{1}^{\prime \prime }\left( x\right) <\frac{1}{30}\frac{t\left(
3019t^{4}+8925t^{3}+12\,705t^{2}+7560t+3780\right) }{\left( t+2\right)
\left( 60t^{4}+154t^{3}+217t^{2}+126t+63\right) }-t+\frac{1}{2}t^{2}
\end{equation*}%
\begin{equation*}
-3t\frac{109\,350\sinh ^{2}t-1890t^{8}\cosh t+5940t^{7}\sinh t+135t^{9}\sinh
t-109\,350t^{2}-t^{14}}{\left( 810\sinh t+t^{7}\right) ^{2}}
\end{equation*}%
\begin{equation*}
:=\frac{810^{2}t\times f_{11}\left( t\right) }{\left( t+2\right) \left(
126t+217t^{2}+154t^{3}+60t^{4}+63\right) \left( 810\sinh t+t^{7}\right) ^{2}}%
,
\end{equation*}%
where%
\begin{equation*}
f_{11}\left( t\right) =p_{6}\left( t\right) \sinh ^{2}t+p_{13}\left(
t\right) \cosh t-p_{14}\left( t\right) \sinh t+p_{20}\left( t\right) ,
\end{equation*}%
\begin{eqnarray*}
p_{6}\left( t\right) &=&30t^{6}+47t^{5}-\frac{718}{15}%
t^{4}-210t^{3}-259t^{2}-\frac{315}{2}t-63, \\
p_{13}\left( t\right) &=&\frac{7}{810}t^{8}\left( t+2\right) \left(
60t^{4}+154t^{3}+217t^{2}+126t+63\right) ,
\end{eqnarray*}%
\begin{equation*}
p_{14}\left( t\right) =t^{7}\left( \frac{1}{27}t^{7}+\frac{77}{810}t^{6}+%
\frac{2857}{1620}t^{5}+\frac{45\,973}{6075}t^{4}+\frac{1547}{108}t^{3}+\frac{%
12\,341}{810}t^{2}+\frac{77}{9}t+\frac{154}{45}\right) ,
\end{equation*}%
\begin{eqnarray*}
p_{20}\left( t\right) &=&\frac{1}{21\,870}t^{20}+\frac{257}{656\,100}t^{19}+%
\frac{13\,667}{9841\,500}t^{18}+\frac{217}{87\,480}t^{17}+\frac{7}{2700}%
t^{16} \\
&&+\frac{7}{4860}t^{15}+\frac{7}{12\,150}t^{14}+30t^{7}+137t^{6}+\frac{525}{2%
}t^{5}+280t^{4}+\frac{315}{2}t^{3}+63t^{2}.
\end{eqnarray*}

To prove $f_{11}\left( t\right) <0$ for $t\in (0,1]$, we use formula $\sinh
^{2}t=\cosh ^{2}t-1$ to write $f_{11}\left( t\right) $ as%
\begin{equation*}
f_{31}\left( t\right) =\left[ p_{6}\left( t\right) \cosh t+p_{13}\left(
t\right) \right] \cosh t-p_{14}\left( t\right) \sinh t+p_{20}\left( t\right)
-p_{6}\left( t\right) .
\end{equation*}%
Since the coefficients of polynomial $-p_{6}\left( t\right) $ satisfy those
conditions of Lemma \ref{L-pz}, and $-p_{6}\left( 1\right) =19\,811/30>0$,
we see that $-p_{6}\left( t\right) >0$ for $t\in (0,1]$. It then follows
from $\cosh t>1$ that%
\begin{equation*}
p_{6}\left( t\right) \cosh t+p_{13}\left( t\right) <p_{6}\left( t\right)
+p_{13}\left( t\right) 
\end{equation*}%
\begin{eqnarray*}
&=&\frac{14}{27}t^{13}+\frac{959}{405}t^{12}+\frac{245}{54}t^{11}+\frac{392}{%
81}t^{10}+\frac{49}{18}t^{9}+\frac{49}{45}t^{8}+30t^{6} \\
&&+47t^{5}-\frac{718}{15}t^{4}-210t^{3}-259t^{2}-\frac{315}{2}%
t-63:=p_{13}^{\ast }\left( t\right) .
\end{eqnarray*}%
Application of Lemma \ref{L-pz} again with $-p_{13}^{\ast }\left( 1\right)
=173\,959/270>0$ yields $-p_{13}^{\ast }\left( t\right) >0$ for $t\in (0,1]$%
, and so $p_{6}\left( t\right) \cosh t+p_{13}\left( t\right) <0$ for $t\in
(0,1]$. Since $p_{14}\left( t\right) >0$ for $t>0$, using the inequalities%
\begin{eqnarray*}
\cosh t &>&\sum_{n=0}^{4}\frac{t^{2n}}{\left( 2n\right) !}=\frac{1}{40\,320}%
t^{8}+\frac{1}{720}t^{6}+\frac{1}{24}t^{4}+\frac{1}{2}t^{2}+1, \\
\sinh t &>&\sum_{n=1}^{4}\frac{t^{2n-1}}{\left( 2n-1\right) !}=\frac{1}{5040}%
t^{7}+\frac{1}{120}t^{5}+\frac{1}{6}t^{3}+t,
\end{eqnarray*}%
we have%
\begin{eqnarray*}
f_{11}\left( t\right)  &=&\left[ p_{6}\left( t\right) \cosh t+p_{13}\left(
t\right) \right] \cosh t-p_{14}\left( t\right) \sinh t+p_{20}\left( t\right)
-p_{6}\left( t\right)  \\
&<&\left[ p_{6}\left( t\right) \sum_{n=0}^{4}\frac{t^{2n}}{\left( 2n\right) !%
}+p_{13}\left( t\right) \right] \sum_{n=0}^{4}\frac{t^{2n}}{\left( 2n\right)
!}-p_{14}\left( t\right) \sum_{n=1}^{4}\frac{t^{2n-1}}{\left( 2n-1\right) !}%
+p_{20}\left( t\right) -p_{6}\left( t\right) 
\end{eqnarray*}%
\begin{eqnarray*}
&=&\frac{1}{54\,190\,080}t^{22}+\frac{9007}{1625\,702\,400}t^{21}+\frac{%
9615\,889}{109\,734\,912\,000}t^{20}+\frac{5351\,449}{9405\,849\,600}t^{19}
\\
&&+\frac{739\,363\,013}{282\,175\,488\,000}t^{18}+\frac{17\,347\,597}{%
2508\,226\,560}t^{17}+\frac{62\,875\,199}{2090\,188\,800}t^{16}+\frac{25\,247%
}{3483\,648}t^{15} \\
&&-\frac{32\,887}{3732\,480}t^{14}-\frac{232\,765}{193\,536}t^{13}-\frac{%
3620\,941}{870\,912}t^{12}-\frac{292\,093}{34\,560}t^{11}-\frac{292\,093}{%
86\,400}t^{10}:=t^{10}p_{12}\left( t\right) ,
\end{eqnarray*}%
From Lemma \ref{L-pz} and $-p_{12}\left( 1\right)
=67\,766\,507\,802\,179/3950\,456\,832\,000>0$ it follows that $%
-p_{12}\left( t\right) >0$ for $t\in (0,1]$, and so $f_{11}\left( t\right) <0
$ for $t\in (0,1]$, which implies $f_{1}^{\prime \prime }\left( x\right) <0$
for $x\geq 1$.

(ii) Using the increasing property of $f_{1}$ and noting that%
\begin{equation*}
f_{1}\left( 1\right) =\ln \frac{e}{\sqrt{2\pi \sinh 1+\pi /405}}\text{ \ and
\ }\lim_{x\rightarrow \infty }f_{1}\left( x\right) =0\text{,}
\end{equation*}%
we have%
\begin{equation*}
\ln \frac{e}{\sqrt{2\pi \sinh 1+\pi /405}}<\ln \frac{\Gamma \left(
x+1\right) }{\sqrt{2\pi x}\left( \frac{x}{e}\right) ^{x}}-\ln \left( x\sinh 
\frac{1}{x}+\frac{1}{810x^{6}}\right) ^{x/2}<0,
\end{equation*}%
which imply the first and second inequalities of (\ref{MI}).

The other ones of (\ref{MI}) follow from Lemma \ref{L-Wk}, which completes
the proof.
\end{proof}

\section{Conclusions}

In this paper, by a little known power series expansion of $\ln \left(
t^{-1}\sinh t\right) $, that is, (\ref{lnsht/t-s}), we establish an
asymptotic expansion (\ref{g-ae-W0}) for gamma function related to
Windschitl's formula, in which its coefficients have a closed-form
expression (\ref{an}). Moreover, we give an estimate of remainder in
asymptotic expansion (\ref{g-ae-W0}) by means of inequalities (\ref{g><})
and (\ref{lnsht/t-s<>}). And due to (\ref{lnsht/t-s}), we also give another
two asymptotic expansions (\ref{g-ae-W0*}), but its coefficients formula is
of recursive form.

Furthermore, we compare accuracy among five approximation formulas for gamma
function generated by truncating five asymptotic series (\ref{g-ae-W0}), (%
\ref{g-ae-W0*}), (\ref{g-ae-W1}), (\ref{g-ae-C}) and (\ref{g-ae-Lu}) by
numeric computations and some inequalities. These show that the
approximation formula (\ref{W1}) is the best.

\section{Acknowledgements.}

The authors would like to express their sincere thanks to the anonymous
referees for their great efforts to improve this paper.

This work was supported by the Fundamental Research Funds for the Central
Universities (No. 2015ZD29) and the Higher School Science Research Funds of
Hebei Province of China (No. Z2015137).

\end{document}